\documentclass{article}
\usepackage{graphicx} 
\usepackage{tikz}
\title{Quantum Plane}

\usepackage[english]{babel}

\usepackage[letterpaper,top=2cm,bottom=2cm,left=3cm,right=3cm,marginparwidth=1.75cm]{geometry}

\usepackage{amsmath, amssymb, latexsym}
\usepackage{graphicx}
\usepackage[colorlinks=true, allcolors=blue]{hyperref}
\usepackage{amsfonts}
\usepackage{amsthm}
\usepackage{color}
\usepackage{indentfirst}
\usepackage{graphicx}%
\usepackage{multirow}%
\usepackage{amsmath,amssymb,amsfonts}%
\usepackage{amsthm}%
\usepackage{mathrsfs}%
\usepackage[title]{appendix}%
\usepackage{xcolor}%
\usepackage{textcomp}%
\usepackage{manyfoot}%
\usepackage{booktabs}%
\usepackage{algorithm}%
\usepackage{algorithmicx}%
\usepackage{algpseudocode}%
\usepackage{listings}%
\usepackage[all]{xy}%
\usepackage{comment}

\usepackage{enumerate,xcolor}
\usepackage{amsrefs}

\newcommand \Q {\mathbb{Q}}
\newcommand \N {\mathbb{N}}

\newcommand \G {\mathbb{G}}
\newcommand \T {\mathbb{T}}
\newcommand \U {\mathcal{U}}
\newcommand \V {\mathcal{V}}
\newcommand \W {\mathcal{W}}
\newcommand \ca {\mathfrak{a}}
\newcommand \pl {\mathfrak{pl}}
\renewcommand \a {A}
\newcommand \ring {R}
\renewcommand \k {K}
\newcommand \affine {\mathbb{A}}
\newcommand \der {{\rm Der}}
\newcommand \saut {{\rm \mathbb{U}}}
\newcommand \auto {{\rm Aut}}
\newcommand \lnd {{\rm LND}}
\newcommand \image {{\rm Im}}
\newcommand \ml {{\rm ML}}
\newcommand \lc {{\rm lc}}

\newtheorem{theorem}{Theorem}

\newtheorem{proposition}[theorem]{Proposition}
\newtheorem{lemma}[theorem]{Lemma}
\newtheorem{example}[theorem]{Example}

\title{\textbf{\LARGE The isotropy group of a derivation on a Danielewski-type algebra}}
\author{ \small
Abdessamad Ahouita, Rene Baltazar, M'hammed El Kahoui and Sergey Gaifullin}
\date{\vspace{-5ex}}

\begin{document}
\Large{
\maketitle

\begin{abstract}Given an algebraically closed field $\k$ of characteristic zero, we consider in this paper $\k$-algebras of the form $$\a_{c,q}=\k[x,y,z]/\big(c(x)z-q(x,y)\big),$$ where $c(x)\in\k[x]$ is a polynomial of degree at least two and $q(x,y)\in\k[x,y]$ is a quasi-monic polynomial of degree at least two with respect to $y$. We give a complete description of the $\k$-automorphism group of $\a_{c,q}$ as an abstract group. Moreover, for every non-locally nilpotent $\k$-derivation $\delta$ of $\a_{c,q}$ we prove that the isotropy group of $\delta$ is a linear algebraic group of dimension at most three.

\end{abstract}

\section{Introduction}\label{sec:intro}
Throughout this paper $\k$ denotes an algebraically closed field of characteristic zero. For a $\k$-algebra $\a$ and a positive integer $n$ we write $\a=\k^{[n]}$ to mean that $\a$ is isomorphic to the $\k$-algebra of polynomials in $n$ variables.

\medskip Let $\a$ be an affine $\k$-algebra, and let $\auto_{\k}(\a)$ and $\der_{\k}(\a)$ be respectively the $\k$-automorphism group and the $\a$-module of $\k$-derivations of $\a$. The group $\auto_{\k}(\a)$ naturally acts on $\der_{\k}(\a)$ by conjugation, and a fundamental question is to classify the $\k$-derivations of $\a$ up to conjugation. A first step towards this task is to describe the {\it isotropy group} of every $\k$-derivation $\delta$ of $\a$, i.e., the subgroup of $\auto_{\k}(\a)$, denoted by $\auto_{\k}(\a,\delta)$, consisting of automorphisms that commute with $\delta$.

\medskip Recall that a $\k$-derivation $\delta$ of $\a$ is said to be {\it locally nilpotent} if for every $a\in\a$ we have $\delta^{k}(a)=0$ for some $k\geq 1$. A classical result due to Rentschler \cite{Rentschler_68} completely classifies, up to conjugation, locally nilpotent derivations of the polynomial algebra $\a=\k^{[2]}$, and allows to easily describe the isotropy group of every such a derivation. In the case the ground field is replaced by a principal ideal domain $\ring$ containing $\Q$, a description of the group $\auto_{\ring}(\a,\delta)$ was recently obtained in \cite{elkahoui25} for every {\it irreducible} locally nilpotent $\ring$-derivation of $\a=\ring^{[2]}$. In dimension three, the isotropy group of a locally nilpotent $\k$-derivation $\delta$ of $\a=\k^{[3]}$ was first described in \cite{finstonWalcher97} under the assumption that $\delta$ is {\it triangularizable}, and then in \cite{stampfli_2014} for the general case. When $\a$ is the coordinate ring of a {\it Danielewski surface} \cite{danielewski_89} the isotropy group of a locally nilpotent $\k$-derivation of $\a$ was described in \cite{baltazarVeloso2020,dasguptaLahiri_2023}.

\medskip In \cite{baltazarPan_2021} the authors studied the isotropy group of an arbitrary $\k$-derivation $\delta$ of $\a=\k^{[2]}$ and obtained, among other things, criteria for $\auto_{\k}(\a,\delta)$ to be an algebraic group. Based on this work it is proved in \cite{pan_2022} that $\auto_{\k}(\a,\delta)$ is an algebraic group if and only if $\delta$ is not locally nilpotent. It is also proved that such a result no longer holds in higher dimensions. Therefore, a natural question is to ask whether the main result in \cite{pan_2022} holds for other types of two-dimensional affine $\k$-algebras. In this paper we address this question for the coordinate rings of a class of Danielewski-type surfaces embedded in the affine $3$-space $\affine_{\k}^3$. More precisely, we consider $\k$-algebras of the form $$\a_{c,q}=\k[x,y,z]/\big(c(x)z-q(x,y)\big),$$ where $c(x)\in\k[x]$ has degree $n\geq 2$ and $q(x,y)\in\k[x,y]$ is {\it quasi-monic} of degree $d\geq 2$ with respect to $y$, i.e., its leading coefficient with respect to $y$ is a constant in $\k^{\star}$. This class of affine algebras includes the ones defined and studied in 
\cite{makar-limanov_2001,bianchi_vesolo_2017}. We first give in Theorem \ref{splittingTheoremCentered} a complete description of $\auto_{\k}(\a_{c,q})$ as an abstract group. We then obtain in Theorem \ref{isotropyGroupTheorem} that for every non-locally nilpotent $\k$-derivation $\delta$ of $\a_{c,q}$ the group $\auto_{\k}(\a_{c,q},\delta)$ is algebraic of dimension at most three. We also give a typical example where $\auto_{\k}(\a_{c,q},\delta)$ is three-dimensional.

\section{Preliminaries}\label{sec:prelim}In this section we recall some basic facts on locally nilpotent derivations theory, and we refer to the books \cite{miyanishi_book,essen_book,freudenburg_book} for more details. Throughout this paper, $\k$ stands for an algebraically closed field of characteristic zero.

\medskip Let $\ring$ be a ring containing $\Q$ and $A$ be an $\ring$-algebra. Given an $\ring$-derivation $\xi$ of $\a$, the kernel of $\xi$ is an $\ring$-subalgebra of $\a$. The $\ring$-derivation $\xi$ is said to be locally nilpotent if every $a\in A$ satisfies $\xi^k(a)=0$ for some $k\geq 1$. It is said to be {\it irreducible} if its image $\xi(A)$ is not contained in any proper principal ideal of $\a$. We denote by $\lnd_{\ring}(A)$ the set of locally nilpotent $\ring$-derivations of $A$. 

A locally nilpotent $\ring$-derivation $\xi$ of $A$ induces an $\ring$-algebra homomorphism $\exp(t\xi):A\longrightarrow A[t]=A^{[1]}$ defined for every $a\in A$ by $$\exp(t\xi)(a)=\sum_{k\geq 0}\frac{1}{k!}\xi^k(a)t^k.$$ If we let $B=\ker(\xi)$, then the extension of $\exp(t\xi)$ to $A[t]$, as an $\ring$-algebra homomorphism, by setting $\exp(t\xi)(t)=t$ yields a $B[t]$-automorphism of $A[t]$ whose inverse is $\exp(-t\xi)$. In particular, for every $a\in B$ the map $\exp(a\xi)$ is a $B$-automorphism of $\a$, called an {\it exponential automorphism} of $\a$. The following result \cite[Proposition 2.1.3]{essen_book} characterizes the exponential automorphisms of $\a$.

\begin{proposition}\label{exponentialAutomorphisms}Let $\a$ be a ring containing $\Q$ and $\sigma$ be an endomorphism of $\a$. Then $\sigma=\exp(\xi)$ for some locally nilpotent derivation $\xi$ of $\a$ if and only if the additive map $\eta=\sigma-{\rm id}_{\a}$ is locally nilpotent. In this case we have
	$$\xi=\sum_{k\geq 1}(-1)^{k+1}\frac{\eta^k}{k}.$$
\end{proposition}

The subset of $B=\ker(\xi)$ defined by $$\pl(\xi)=\big\{\xi(s)\mid \xi^2(s)=0\big\}$$ is an ideal of $B$ called the {\it plinth ideal} of $\xi$. For every nonzero $a\in A$ we define the {\it degree} of $a$ with respect to $\xi$, denoted by $\deg_{\xi}(a)$, to be the degree of the polynomial $\exp(t\xi)(a)$ with respect to $t$. When $A$ is an integral domain $\deg_{\xi}$ is a degree function, i.e., for every nonzero $a,b\in A$ we have $$\deg_{\xi}(ab)=\deg_{\xi}(a)+\deg_{\xi}(b).$$ In particular, we obtain that $B$ is {\it factorially closed} in $A$. That is for every $a,b\in A$ such that $ab\in B\setminus\{0\}$ we have $a,b\in B$.

\medskip Throughout this paper we will consider $\k$-algebras of the form $$\a_{c,q}=\k[x,y,z]/\big(c(x)z-q(x,y)\big)=\k[\bar{x},\bar{y},\bar{z}],$$ where $c(x)\in\k[x]$ has degree $n\geq 2$ and $q(x,y)\in\k[x,y]$ is quasi-monic of degree $d\geq 2$ with respect to $y$. Such an algebra is the coordinate ring of the Danielewski surface, embedded in the affine space $\affine_{\k}^{3}$, defined by the equation $c(x)z-q(x,y)=0$ and will be called a {\it Danielewski algebra}. The $\k$-algebra $\a_{c,q}$ is an integral domain and can naturally be endowed with a locally nilpotent $\k$-derivation denoted $\xi_{c,q}$ and defined by $$\xi(\bar{x})=0,\quad \xi_{c,q}(\bar{y})=c(\bar{x}),\quad \xi_{c,q}(\bar{z})=\partial_yq(\bar{x},\bar{y}).$$ The derivation $\xi_{c,q}$ will be called the {\it standard derivation} of $\a_{c,q}$. The following result \cite[Theorem 5]{bianchi_vesolo_2017} gathers some fundamental properties of Danielewski algebras.

\begin{theorem}\label{basisLemma}Let $c(x)\in\k[x]$ be a non-constant polynomial and $q(x,y)\in\k[x,y]$ be a quasi-monic polynomial of degree $d\geq 2$ with respect to $y$. Then $\ker(\xi_{c,q})=\k[\bar{x}]$ and $\lnd_{\k[\bar{x}]}(\a_{c,q})=\k[\bar{x}]\xi_{c,q}$. In particular, $\xi_{c,q}$ is irreducible.
\end{theorem}

Let $\a$ be an integral domain containing $\Q$ endowed with a nonzero locally nilpotent derivation $\xi$ and let $B=\ker(\xi)$. Following the terminology introduced in \cite{freudenburg_2019}, we will say that $(\a,\xi)$ has the
{\it freeness property} if $\a$ is a free $B$-module and has a basis $(s_k)_{k\in\N}$ such that $\deg_{\xi}(s_k)=k$ for every $k\in\N$. Such a basis is called a $\xi$-basis of $\a$ over $B$. The following is a particular case of \cite[Theorem 3.3]{elkahoui_2023}.

\begin{proposition}\label{xbasisProposition}Let $c\in\k[x]$ be a non-constant polynomial and $q(x,y)\in\k[x,y]$ be a quasi-monic polynomial of degree $d\geq 2$ with respect to $y$. Then 
	$$\big\{ \bar{y}^i\bar{z}^j\mid 0\leq i<d, \;j\in\N\}$$ is a $\xi_{c,q}$-basis of $\a_{c,q}$ over $\k[\bar{x}]=\ker(\xi_{c,q})$.
\end{proposition}

Let $\ring$ be a ring and $p(y)\in\ring[y]$ be a quasi-monic polynomial of degree $d\geq 1$. We will say that $p(y)$ is {\it centered} if its coefficient of degree $d-1$ is zero. Now assume that $\ring=\k[x]$ and let $c(x)\in\k[x]$ be non-constant and $q(x,y)\in\k[x,y]$ be quasi-monic of degree $d\geq 2$ with respect to $y$. Borrowing the terminology of \cite[Definition 4]{poloni_2011}, the Danielewski surface defined by $c(x)z-q(x,y)=0$ is said to be in {\it reduced form} if $c(x)$ is centered, $q(x,y)$ is centered as a polynomial in $\k[x][y]$ and $\deg_x(q)<\deg(c)$. In this case, we will also say that the Danielewski algebra $\a_{c,q}$ is in reduced form. The following lemma \cite[Lemma 3.1]{elkahouiHammi_2025} is an adaptation of \cite[Lemma 2]{poloni_2011} to our situation.


\begin{proposition}\label{standardForm}Let $c(x)\in\k[x]$ be non-constant and $q(x,y)\in\k[x,y]$ be quasi-monic of degree $d\geq 2$ with respect to $y$. Then there exists a triangular $\k$-automorphism $\sigma$ of $\k[x,y,z]$ such that $$\sigma\big (c(x)z-q(x,y)\big)=\tilde{c}(x)z-\tilde{q}(x,y),$$ where 
	\begin{enumerate}[a.]
		\item $\tilde{c}(x)$ is centered and $\deg(\tilde{c})=\deg(c)$,
		\item $\tilde{q}(x,y)$ is quasi-monic, centered of degree $d$ with respect to $y$ and $\deg_x(\tilde{q})<\deg(\tilde{c})$.
	\end{enumerate}
	Therefore, the $\k$-algebra $\a_{\tilde{c},\tilde{q}}$ is in reduced form and $\sigma$ induces a differential $\k$-algebra isomorphism from $(\a_{c,q},\xi_{c,q})$ onto $(\a_{\tilde{c},\tilde{q}},\xi_{\tilde{c},\tilde{q}})$.
\end{proposition}

\section[Automorphism group of a Danielewski algebra]{Structure of the automorphism group of a Danielewski algebra}\label{sec:automorphismsDanielewski}Recall that the {\it Makar-Limanov invariant} of a $\k$-algebra $\a$, denoted by $\ml(\a)$, is the intersection of the kernels of all locally nilpotent $\k$-derivations of $\a$. When $\deg(c)\geq 2$ the Makar-Limanov invariant of $\a_{c,q}$ is $\ml(\a_{c,q})=\k[\bar{x}]$. This property is proved in \cite{makar-limanov_2001,poloni_2011} for the case $c(x)=x^n$ and then for the general case in \cite{bianchi_vesolo_2017}.

\begin{lemma}\label{decompositionLemma}Let $c(x)\in\k[x]$ be a polynomial of degree $n\geq 2$ and $q(x,y)\in \k[x,y]$ be a quasi-monic polynomial of degree $d\geq 2$ with respect to $y$. Then we have the following properties.
	\begin{enumerate}
		\item Every $\sigma\in\auto_{\k}(\a_{c,q})$ satisfies $$\sigma\big(\k[\bar{x}]\big)=\k[\bar{x}],\quad \sigma\big(\pl(\xi_{c,q})\big)=\pl(\xi_{c,q}).$$
		\item For every $\sigma\in\auto_{\k}(\a_{c,q})$ there exists a unique $(u_{\sigma},h_{\sigma})\in\k^{\star}\times\k[x]$ such that $\sigma(\bar{y})=u_{\sigma}\bar{y}+h_{\sigma}(\bar{x})$.
	\end{enumerate}
\end{lemma}
\begin{proof}
	\smallskip {\it 1.} The fact that $\ml(\a_{c,q})=\k[\bar{x}]$ yields $\lnd_{\k}(\a_{c,q})=\lnd_{\k[\bar{x}]}(\a_{c,q})$. Thus, for every $\sigma\in\auto_{\k}(\a_{c,q})$ the locally nilpotent derivation $\sigma\xi_{c,q}\sigma^{-1}$ is irreducible, and hence $\sigma\xi_{c,q}\sigma^{-1}=e\xi_{c,q}$ for some $e\in\k^{\star}$. The two claimed properties immediately follow from the fact that $\ker(\sigma\xi_{c,q}\sigma^{-1})=\sigma\big(\ker(\xi_{c,q})\big)$ and $\pl(\sigma\xi_{c,q}\sigma^{-1})=\sigma\big(\pl(\xi_{c,q})\big)$.
	
	\smallskip {\it 2.} Let $\sigma\in\auto_{\k}(\a_{c,q})$ and write $\sigma\xi_{c,q}\sigma^{-1}=e\xi_{c,q}$ for some $e\in\k^{\star}$. A direct computation shows that $\xi_{c,q}\big(\sigma(\bar{y})\big)=e^{-1}\sigma\big(c(\bar{x})\big)$. The fact that $\pl(\sigma\xi_{c,q}\sigma^{-1})=\sigma\big(\pl(\xi_{c,q})\big)$ and that $c(\bar{x})$ generates $\pl(\xi_{c,q})$ then imply that $e^{-1}\sigma\big(c(\bar{x})\big)=u_{\sigma}c(\bar{x})$ for some uniquely determined $u_{\sigma}\in\k^{\star}$. Thus, we have $\xi_{c,q}\big(\sigma(\bar{y})-u_{\sigma}\bar{y}\big)=0$, and hence there exists a unique $h_{\sigma}(x)\in\k[x]$ such that $\sigma(\bar{y})=u_{\sigma}\bar{y}+h_{\sigma}(\bar{x})$.
\end{proof}

\smallskip By Lemma \ref{decompositionLemma}, every automorphism $\sigma\in\auto_{\k}(\a_{c,q})$ leaves $\k[\bar{x}]$ and $\pl(\xi_{c,q})$ invariant and so there exists a unique pair $(e_{\sigma},a_{\sigma})\in\k^{\star}\times\k$ such that $\sigma(\bar{x})=e_{\sigma}\bar{x}+a_{\sigma}$. Lemma \ref{decompositionLemma} also implies that there exists a unique pair $(u_{\sigma},h_{\sigma})\in\k^{\star}\times\k[x]$ such that $\sigma(\bar{y})=u_{\sigma}\bar{y}+h_{\sigma}(\bar{x})$. This yields a map $$\psi:\auto_{\k}(\a_{c,q})\longrightarrow \k^{\star}\times\k^{\star}$$ defined for every $\sigma\in\auto_{\k}(\a_{c,q})$ by $\psi(\sigma)=(e_{\sigma},u_{\sigma})$. As we will see in the following lemma, $\psi$ is a group homomorphism. We will call it the {\it canonical homomorphism} from $\auto_{\k}(\a_{c,q})$ to $\k^{\star}\times\k^{\star}$. The {\it unipotent} $\k$-automorphism group of $\a_{c,q}$, denoted by $\saut_{\k}(\a_{c,q})$, is the subgroup of $\auto_{\k}(\a_{c,q})$ generated by $\exp(\delta)$, where $\delta$ ranges over $\lnd_{\k}(\a_{c,q})$.

\begin{lemma}\label{canonicalHomomorphism}Let $c(x)\in\k[x]$ be a polynomial of degree $n\geq 2$ and $q(x,y)\in\k[x,y]$ be a quasi-monic polynomial of degree $d\geq 2$ with respect to $y$. Then the map $$\psi:\sigma\in \auto_{\k}(\a_{c,q})\longrightarrow (e_{\sigma},u_{\sigma})\in \k^{\star}\times\k^{\star}$$ is a group homomorphism and its kernel is $\saut_{\k}(\a_{c,q})$. In other words, we have the following short exact sequence $$\xymatrix{1 \ar[r] & \saut_{\k}(\a_{c,q}) \ar[r]^-{} & \auto_{\k}(\a_{c,q}) \ar[r]^-{\psi} & \image(\psi) \ar[r] & 1.}$$ 
\end{lemma}
\begin{proof}Let $\sigma_1,\sigma_2\in\auto_{\k}(\a_{c,q})$ and let us write $\sigma_i(\bar{x})=e_i\bar{x}+a_i$ and $\sigma_i(\bar{y})=u_i\bar{y}+h_i$, where $(e_i,a_i)\in\k^{\star}\times \k$ and $(u_i,h_i)\in\k^{\star}\times \k[\bar{x}]$. A straightforward computation shows that
	$$\begin{array}{ll}
		\sigma_1\sigma_2(\bar{x})=e_1e_2\bar{x}+e_2a_1+a_2, \\\sigma_1\sigma_2(\bar{y})=u_1u_2\bar{y}+u_2h_1+\sigma_1(h_2).
	\end{array}$$
	Therefore, $\psi$ is a group homomorphism. Now let $\sigma\in \ker(\psi)$ and let us first show that $\sigma(\bar{x})=\bar{x}$. We have $\sigma(\bar{x})=\bar{x}+a$ for some $a\in\k$, and since $\sigma(\pl(\xi_{c,q}))=\pl(\xi_{c,q})$ and $c(\bar{x})$ generates $\pl(\xi_{c,q})$ it follows that $c(\bar{x}+a)=ec(\bar{x})$ for some $e\in\k^{\star}$ and since $\k[\bar{x}]=\k^{[1]}$ we have $c(x+a)=ec(x)$. By Taylor-expanding we obtain $$ec(x)=\sum_{k=0}^{n}\frac{a^k}{k!}c^{(k)}(x).$$ Since $c(x),c^{\prime}(x),\ldots,c^{(n)}(x)$ are linearly independent over $\k$ we actually have $a=0$, and hence $\sigma$ induces the identity on $\k[\bar{x}]$.
	
	\smallskip The fact that $\sigma\in \ker(\psi)$ implies that $u_{\sigma}=1$ and so $\sigma(\bar{y})=\bar{y}+h$, where $h\in\k[\bar{x}]$. To prove that $\sigma$ is an exponential automorphism it suffices, by Proposition \ref{exponentialAutomorphisms}, to show that the additive map $\eta=\sigma-{\rm id}_{\a_{c,q}}$ is locally nilpotent. Since $\sigma(a)=a$ for every $a\in\k[\bar{x}]$ it follows that $\eta$ is actually a $\k[\bar{x}]$-module homomorphism. As a consequence, for every nonzero $p(x,y)\in\k[x,y]$ we have $\eta\big(p(\bar{x},\bar{y})\big)=p(\bar{x},\bar{y}+h)-p(\bar{x},\bar{y})$ and then $\eta^{m+1}\big(p(\bar{x},\bar{y})\big)=0$, where $m=\deg_y(p)$. Given any nonzero $g\in\a_{c,q}$ we can find a nonzero $a(x)\in\k[x]$ and $p(x,y)\in\k[x,y]$ such that $a(\bar{x})g=p(\bar{x},\bar{y})$. This gives $a(\bar{x})\eta^{m+1}(g)=0$, where $m=\deg_y(p)$, and then $\eta^{m+1}(g)=0$.
\end{proof} 

Our aim in the rest of this section is to prove that the short exact sequence $$\xymatrix{1 \ar[r] & \saut_{\k}(\a_{c,q}) \ar[r]^-{} & \auto_{\k}(\a_{c,q}) \ar[r]^-{\psi} & \image(\psi) \ar[r] & 1.}$$ is split exact. Let $\G_a$ and $\G_m$ stand for respectively the additive group $(\k,+)$ and the multiplicative group $(\k^{\star},\times)$ of $\k$.
For every positive integer $r$ let $\T_r=\G_m^r$ be the $r$-dimensional algebraic torus over $\k$. By Proposition \ref{standardForm} we may assume without loss of generality that $\a_{c,q}$ is in reduced form. Let us consider the subgroup of $\T_2$ defined by
$$\G_{c,q}=\Big\{(e,u)\in\T_2\mid c(ex)=e^nc(x),\; q(ex,uy)=u^dq(x,y) \Big\}.$$Clearly, $\G_{c,q}$ is an algebraic subgroup of the algebraic torus $\T_2$. For every $(e,u)\in\G_{c,q}$ let $\phi_{e,u}$ be the $\k$-automorphism of $\k[x,y,z]$ defined by $$\phi_{e,u}(x)=ex,\quad \phi_{e,u}(y)=uy,\quad \phi_{e,u}(z)=e^{-n}u^dz.$$ An easy computation shows that $$\phi_{e,u}\big(c(x)z-q(x,y)\big)=u^d\big(c(x)z-q(x,y)\big)$$ and hence $\phi_{e,u}$ induces a $\k$-automorphism $\varphi_{e,u}\in\auto_{\k}(\a_{c,q})$. We have thus constructed a map $$\varphi: (e,u)\in \G_{c,q}\longmapsto \varphi_{e,u}\in \auto_{\k}(\a_{c,q}).$$ Clearly, the map $\varphi$ is an injective group homomorphism. We will call it the {\it canonical homomorphism} from $\G_{c,q}$ into $ \auto_{\k}(\a_{c,q})$.

\begin{theorem}\label{splittingTheoremCentered}Let $c(x)\in\k[x]$ be a polynomial of degree $n\geq 2$ and $q(x,y)\in\k[x,y]$ be a quasi-monic polynomial of degree $d\geq 2$ with respect to $y$. Assume that $\a_{c,q}$ is in reduced form and let $\psi:\auto_{\k}(\a_{c,q})\longrightarrow \T_2$ and $\varphi:\G_{c,q}\longrightarrow\auto_{\k}(\a_{c,q})$ be the canonical homomorphisms. Then we have $\image(\psi)=\G_{c,q}$ and $\psi\varphi={\rm id}_{\G_{c,q}}$. As a consequence, the short exact sequence
	$$\xymatrix{1 \ar[r] & \saut_{\k}(\a_{c,q}) \ar[r]^-{} & \auto_{\k}(\a_{c,q}) \ar[r]^-{\psi} & \image(\psi) \ar[r] & 1}$$ is split and hence the group $\auto_{\k}(\a_{c,q})$ is the inner semidirect product of $\saut_{\k}(\a_{c,q})$ and $\image(\varphi)$.
\end{theorem}
\begin{proof}Clearly, we have $\psi\varphi={\rm id}_{\G_{c,q}}$ and hence $\G_{c,q}\subseteq \image(\psi)$. Now let $(e,u)\in\image(\psi)$ and let $\sigma\in\auto_{\k}(\a_{c,q})$ be such that $\sigma(\bar{x})=e\bar{x}+a$ and $\sigma(\bar{y})=u\bar{y}+h(\bar{x})$ for some $a\in\k$ and $h\in\k[x]$. Since $\sigma\big(c(\bar{x})\big)=e^nc(\bar{x})$ and $c(x)$ is centered it follows immediately that $a=0$.
	
	\smallskip Let us now show that $c(x)$ divides $h(x)$. Let $f(x,y) \in \k[x,y]$ be the polynomial defined by $$f(x,y)=q\big(ex,uy+h(x)\big)-u^dq(x,y)$$ and notice that, since $q(x,y)$ is quasi-monic with respect to $y$, we have $\deg_y(f)<d$. Since moreover $q(x,y)$, viewed in $\k[x][y]$, is centered the coefficient of degree $d-1$ with respect to $y$ of $f(x,y)$ is $du^{d-1}h$. On the other hand, from the equality $c(\bar{x})\bar{z}=q(\bar{x},\bar{y})$ in $\a_{c,q}$ it follows that $$c(e\bar{x})\sigma(\bar{z})=q\big(e\bar{x},u\bar{y}+h(\bar{x})\big)=u^dc(\bar{x})\bar{z}+f(\bar{x},\bar{y}),$$ 
	and since $c(e\bar{x})=e^nc(\bar{x})$ it follows that $c(\bar{x})$ divides $f(\bar{x},\bar{y})$ in $\a_{c,q}$. By Proposition \ref{xbasisProposition}, there exists a unique polynomial $g\in \k[x,y,z]$ such that $\deg_y(g)<d$ and $f(\bar{x},\bar{y})=cg(\bar{x},\bar{y},\bar{z})$ in $\a_{c,q}$. The fact that $\deg_y(f)<d$, $\deg_y(g)<d$ and $f\in\k[x,y]$ then imply that $g\in\k[x,y]$ and $f(x,y)=cg(x,y)$. In particular, $du^{d-1}h(\bar{x})=c(\bar{x})g_{d-1}(\bar{x})$, where $g_{d-1}$ is the coefficient of degree $d-1$ of $g$ with respect to $y$. Since $d,u\in\k^{\star}$ it follows that $c(\bar{x})$ divides $h(\bar{x})$ in $\k[\bar{x}]$. Thus, we have $q\big(ex,uy+h(x)\big)-u^dq(x,y)=cg(x,y)$ and $c(x)$ divides $h(x)$, and this clearly implies that $$q(ex,uy)=u^dq(x,y)\mod c(x).$$ Since on the other hand $\deg_x(q)<\deg(c)$ we actually have $q(ex,uy)=u^dq(x,y)$ and hence $(e,u)\in\G_{c,q}$. We have thus proved that $\image(\psi)=\G_{c,q}$, and since moreover $\psi\varphi={\rm id}_{\G_{c,q}}$ the exact sequence $$\xymatrix{1 \ar[r] & \saut_{\k}(\a_{c,q}) \ar[r]^-{} & \auto_{\k}(\a_{c,q}) \ar[r]^-{\psi} & \image(\psi) \ar[r] & 1}$$ is split. Therefore, $\auto_{\k}(\a_{c,q})$ is the inner semidirect product of $\saut_{\k}(\a_{c,q})$ and $\image(\varphi)=\varphi\big(\G_{c,q}\big)$.
\end{proof}

\section{Isotropy group of a derivation of $\a_{c,q}$}
Throughout this section we let $c(x)\in\k[x]$ be a polynomial of degree $n\geq 2$ and $q(x,y)\in\k[x,y]$ be a quasi-monic polynomial of degree $d\geq 2$ with respect to $y$, and assume that the $\k$-algebra $\a_{c,q}$ is in reduced form. Given a $\k$-derivation $\delta$ of $\a_{c,q}$ we let $\auto_{\k}(\a_{c,q},\delta)$ be the subgroup of $\auto_{\k}(\a_{c,q})$ consisting of automorphisms that commute with $\delta$. In this section we study the structure of $\auto_{\k}(\a_{c,q},\delta)$. Our main result is the following Theorem.
\begin{theorem}\label{isotropyGroupTheorem}Let $\delta$ be a non-locally nilpotent derivation of $\a_{c,q}$. Then  $\auto_{\k}(\a_{c,q},\delta)$ is a closed algebraic subgroup of $\auto_{\k}(\a_{c,q})$ of dimension at most $3$. Moreover, it falls into one of the following two cases.
	\begin{enumerate}[1.]
		\item $\auto_{\k}(\a_{c,q},\delta)$ is isomorphic to a closed subgroup of $\G_{c,q}$.
		\item $\auto_{\k}(\a_{c,q},\delta)$ is the semi-direct product of a closed algebraic subgroup of $\saut_{\k}(\a_{c,q})$ isomorphic to the additive group $\G_a$ of the field $\k$ and a closed algebraic subgroup of $\G_{c,q}$.
	\end{enumerate}
\end{theorem}
In order to prove Theorem \ref{isotropyGroupTheorem} we first give a digest on the structure of $\auto_{\k}(\a_{c,q})$ as an {\it ind-group} introduced by Shafarevich \cite{shafarevich_1966}, and we refer to \cite{zaidenberg_2017,furterKraft_2018} for a comprehensive treatment.

\medskip Throughout, all affine varieties are viewed over $\k$ and are equipped with the Zariski topology. An {\it affine ind-variety} is a set $\V$ together with a sequence $$\V_1\subseteq \V_2\subseteq \cdots \V_k\subseteq \cdots$$ of affine varieties such that $$\V=\displaystyle{\bigcup_{k\in\N}\V_k}$$ and for every positive integer $k$ the inclusion map $\V_k\hookrightarrow\V_{k+1}$ is a closed embedding of affine varieties. An affine ind-variety $\V$ carries a natural topology, namely the {\it Zariski topology}, where a subset $\U$ of $\V$ is closed if and only if for every positive integer $k$ the intersection $\U\cap \V_k$ is a closed subset of $\V_k$. In this case, $\U$ can be endowed with an affine ind-variety structure through the sequence $$\U\cap\V_1\subseteq \U\cap\V_2\subseteq \cdots\subseteq \U\cap\V_k\subseteq \cdots.$$ We will say that $\U$ is a closed {\it affine ind-subvariety} of $\V$. A {\it filtration} on an affine ind-variety $\V$ is a sequence $$\W_1\subseteq \W_2\subseteq \cdots \subseteq \W_k\subseteq \cdots$$ of closed subsets of $\V$ such that $\V=\bigcup\W_k$. The filtration is said to be {\it admissible} if it defines an equivalent ind-variety structure on $\V$, i.e., for very $k\geq 1$ there exist $\ell_1,\ell_2\geq 1$ such that $\V_k\subseteq \W_{\ell_1}$ and $\W_k\subseteq \V_{\ell_2}$.

\smallskip Given two affine ind-varieties $\V=\bigcup \V_k$ and $\W=\bigcup\W_{k}$, a map $\varphi:\V\longrightarrow\W$ is an {\it ind-morphism} if for every $k\geq 1$ there exists $\ell\geq 1$ such that $\varphi(\V_k)\subseteq \W_{\ell}$ and the induced map $\V_k\longrightarrow\W_{\ell}$ is a morphism of affine varieties. The product $\V\times \W$ is naturally endowed with an affine ind-variety structure through the filtration $$\V_1\times \W_1\subseteq \V_2\times\W_2\subseteq \cdots \subseteq \V_k\times \W_k\subseteq \cdots.$$
An {\it affine ind-group} $\mathcal{G}$ is an affine ind-variety endowed with a group structure such that the map $(g,h)\in \mathcal{G}\times \mathcal{G}\longrightarrow gh^{-1}\in \mathcal{G}$ is an ind-morphism of affine ind-varieties. A closed subgroup $\mathcal{H}$ of $\mathcal{G}$ endowed with its ind-subvariety structure is an ind-group that will be called a closed {\it ind-subgroup} of $\mathcal{G}$. Given two affine ind-groups $\mathcal{G}$ and $\mathcal{H}$, a group homomorphism $\varphi:\mathcal{G}\longrightarrow \mathcal{H}$ which is also an ind-morphism is called an {\it ind-group morphism}.

\smallskip Let $\a$ be a finitely generated $\k$-algebra, and let us sketch how to endow the abstract group $\auto_{\k}(\a)$ with a natural affine ind-group structure. We refer to \cite[Proposition 2.1]{zaidenberg_2017} and \cite[Section 5]{furterKraft_2018} for more details.

\smallskip\noindent Let $\a=\k[x_1,\ldots,x_r]/\ca$, where $\ca$ is an ideal of $\k[x_1,\ldots,x_r]$, be a presentation of $\a$ and $$\pi:\k[x_1,\ldots,x_r]\longrightarrow\a$$ be the canonical epimorphism. For every $m\geq 1$ let
$$\W_m=\big\{\pi(p)\mid p\in\k[x_1,\ldots,x_r]\; {\rm and}\; \deg(p)\leq m \big\},$$ and
$$\mathcal{V}_m=\big\{\sigma\in\auto_{\k}(\a)\mid \sigma(\bar{x}_i), \sigma^{-1}(\bar{x}_i)\in\W_m\;{\rm for}\; i=1,\ldots,r \big\}.$$ Then the sequence $\mathcal{V}_1\subseteq \mathcal{V}_2\subseteq \cdots\subseteq \mathcal{V}_m\subseteq \cdots$ endows $\auto_{\k}(\a)$ with an affine ind-group structure which actually does not depend on the chosen presentation of $\a$.

\medskip Let $\psi:\auto_{\k}(\a_{c,q})\longrightarrow \T_2$ and $\varphi:\G_{c,q}\longrightarrow\auto_{\k}(\a_{c,q})$ be the canonical homomorphisms. By Theorem \ref{splittingTheoremCentered}, for every $\sigma\in\auto_{\k}(\a_{c,q})$ there exist unique $(e_{\sigma},u_{\sigma})\in\G_{c,q}$ and $\ell_{\sigma}\in\k[x]$ such that $\sigma(\bar{x})=e_{\sigma}\bar{x}$ and $\sigma(\bar{y})=u_{\sigma}\bar{y}+c(\bar{x})\ell_{\sigma}(\bar{x})$. For every $m\geq 1$ let
$$\mathcal{G}_m=\Big\{\sigma\in\auto_{\k}(\a_{c,q})\mid \deg(\ell_{\sigma})\leq m \Big\}.$$ Clearly, every $\mathcal{G}_m$ is an affine variety isomorphic to $\G_{c,q}\times \affine_{\k}^{m+1}$, and also a subgroup of $\auto_{\k}(\a_{c,q})$. Moreover, $\auto_{\k}(\a_{c,q})$ is an increasing union of the $\mathcal{G}_m$'s.

\begin{lemma}\label{nestedStructure}The $\mathcal{G}_m$'s form an admissible filtration of $\auto_{\k}(\a_{c,q})$. Moreover, $\psi$ and $\varphi$ are morphisms of ind-groups.
\end{lemma}
\begin{proof}Let $\pi:\k[x,y,z]\longrightarrow\a_{c,q}$ be the canonical epimorphism, and for every $m\geq 1$ let $$\W_m=\big\{\pi(p)\mid \deg(p)\leq m \big\}$$ and $$\V_m=\Big\{\sigma\in\auto_{\k}(\a_{c,q})\mid \sigma(v),\sigma^{-1}(v)\in\W_m\; {\rm for}\; v=\bar{x},\bar{y},\bar{z} \Big\}.$$ Then the sequence $\V_1\subseteq \V_2\subseteq \cdots\subseteq \V_m\subseteq \cdots $ is an admissible filtration of the ind-group $\auto_{\k}(\a_{c,q})$. On the other hand, for every $m\geq 1$ and every $\sigma\in\mathcal{G}_m$ we have $\sigma(\bar{z})\in\W_{(m+n)d}$ and hence $\V_m\subseteq \mathcal{G}_m\subseteq \V_{(m+n)d}$. Moreover, we have $$\mathcal{G}_m=\Big\{ \sigma\in\V_{(m+n)d}\mid \sigma(\bar{y})\in\W_{m+n}\Big\}$$ which shows that $\mathcal{G}_m$ is a closed subset of $\V_{(m+n)d}$. Therefore, the $\mathcal{G}_m$'s form an admissible filtration of $\auto_{\k}(\a_{c,q})$.
	
	\smallskip It remains to prove that $\varphi$ and $\psi$ are ind-group morphisms. Notice first that $\varphi(\G_{c,q})\subseteq \mathcal{G}_1$ and the map $(e,u)\in\G_{c,q}\longrightarrow\varphi_{e,u}\in\mathcal{G}_1$ is a morphism of affine algebraic groups. Therefore, $\varphi$ is an ind-group morphism and $\image(\varphi)$ is a closed algebraic subgroup of $\auto_{\k}(\a_{c,q})$. On the other hand, since for every $m\geq 1$ the affine variety $\mathcal{G}_m$ is isomorphic to $\G_{c,q}\times \affine_{\k}^{m+1}$ it follows that the restriction of $\psi$ to $\mathcal{G}_m$ is a morphism of affine varieties. This proves that $\psi$ is an ind-group morphism.
\end{proof}

\begin{lemma}\label{closedSubgroup}Let $\delta$ be a $\k$-derivation of $\a_{c,q}$. Then $\auto_{\k}(\a_{c,q},\delta)$ is a closed ind-subgroup of $\auto_{\k}(\a_{c,q})$. If moreover $\delta$ is not locally nilpotent then $\auto_{\k}(\a_{c,q},\delta)$ is an algebraic subgroup of $\auto_{\k}(\a_{c,q})$.
\end{lemma}
\begin{proof}In order to prove that $\auto_{\k}(\a_{c,q},\delta)$ is a closed subgroup of $\auto_{\k}(\a_{c,q})$ we need to show that $\auto_{\k}(\a_{c,q},\delta)\cap \mathcal{G}_m$ is closed in $\mathcal{G}_m$ for every $m\geq 0$, where $$\mathcal{G}_m=\big\{ \sigma\in\auto_{\k}(\a_{c,q})\mid \deg(\ell)\leq m\big\},$$ and $\ell(x)\in\k[x]$ is such that $\sigma(\bar{y})=u\bar{y}+c(\bar{x})\ell(\bar{x})$.
	
	\medskip Let $k\geq 0$ be such that \begin{equation}\label{cancellationOfZ}c^{k}(\bar{x})\delta(\bar{x})=f_1(\bar{x},\bar{y}) \quad \text{and} \quad c^{k}(\bar{x})\delta(\bar{y})=f_2(\bar{x},\bar{y}),\end{equation} where $f_i\in \k[x,y]=K^{[2]}$ depends only on $\delta$. Let $\sigma\in\mathcal{G}_m$ and write $\sigma(\bar{x})=e\bar{x}$ and $\sigma(\bar{y})=u\bar{y}+h(\bar{x})$, where $(e,u)\in\G_{c,q}$ and $h(x)=c(x)\ell(x)$ and $\deg(\ell)\leq m$. Then by applying $\sigma$ to the equalities in (\ref{cancellationOfZ}), and taking into account $c(ex)=e^nc(x)$, we have $$\begin{array}{ccc}f_1(e\bar{x},u\bar{y}+h(\bar{x}))&=&
		e^{kn}c^k(\bar{x})\sigma\delta(\bar{x}),\\
		f_2(e\bar{x},u\bar{y}+h(\bar{x}))&=&e^{kn}c^k(\bar{x})\sigma\delta(\bar{y}).\end{array}$$
	On the other hand, we have
	$$\begin{array}{ccc}
		e^{kn}c^k(\bar{x})\delta\sigma(\bar{x})&=&e^{kn+1}f_1(\bar{x},\bar{y}),\\
		e^{kn}c^k(\bar{x})\delta\sigma(\bar{y})&=&e^{kn}f_2(\bar{x},\bar{y})+e^{kn}h^{\prime}(\bar{x})f_1(\bar{x},\bar{y}).
	\end{array}$$
	Since $\a_{c,q}$ is an integral domain and $c(\bar{x})\bar{z}=q(\bar{x},\bar{y})$ it follows that $\sigma\delta=\delta\sigma$ if and only if
	\begin{equation}\label{firstCondition}f_1(e\bar{x},u\bar{y}+h(\bar{x}))-e^{kn+1}f_1(\bar{x},\bar{y})=0
	\end{equation}
	and
	\begin{equation}\label{secondCondition}f_2(e\bar{x},u\bar{y}+h(\bar{x}))-e^{kn}uf_2(\bar{x},\bar{y})-e^{kn}h^{\prime}(\bar{x})f_1(\bar{x},\bar{y})=0.
	\end{equation}
	Since moreover $\k[\bar{x},\bar{y}]=\k^{[2]}$ the equations (\ref{firstCondition}) and (\ref{secondCondition}) yield a system of polynomial equations in terms of $e,u$ and the coefficients of $\ell(x)$, where $h(x)=c(x)\ell(x)$. This proves that $\auto_{\k}(\a_{c,q},\delta)\cap \mathcal{G}_m$ is closed in $\mathcal{G}_m$, and since this holds for every $m\geq 0$ it follows that $\auto_{\k}(\a_{c,q},\delta)$ is a closed ind-subgroup of $\auto_{\k}(\a_{c,q})$.
	
	\medskip Assume now that $\delta$ is not locally nilpotent. In order to prove that $\auto_{\k}(\a_{c,q},\delta)$ is an algebraic subgroup of $\auto_{\k}(\a_{c,q})$ we only need to show that $\auto_{\k}(\a_{c,q},\delta)\subseteq \mathcal{G}_m$ for some positive integer $m$. That is, for every $\sigma\in\auto_{\k}(\a_{c,q},\delta)$, with $\sigma(\bar{y})=u\bar{y}+h(\bar{x})$, we have $\deg(h)\leq m+n$ where $n=\deg(c)$.
	
	\medskip Let $\sigma\in\auto_{\k}(\a_{c,q},\delta)$, and let $(e,u)\in \G_{c,q}$ and $h(x)\in c(x)\k[x]$ be such that $$\sigma(\bar{x})=e\bar{x},\quad \sigma(\bar{y})=u\bar{y}+h(\bar{x}).$$ Then we have several possibilities, depending on the nature of $f_1$ and $f_2$.
	
	\medskip -- The case $f_1(x,y)\in K[x,y] \setminus \k[x]$. From the equation (\ref{firstCondition}) we have
	$$	f_1(e\bar{x},u\bar{y}+h(\bar{x}))= e^{kn+1} f_1(\bar{x}, \bar{y})$$
	and since $\k[\bar{x},\bar{y}]=K^{[2]}$ we actually have \begin{equation}\label{conditioncommefromxx} f_1(ex,uy+h(x))=e^{kn+1}f_1(x,y).\end{equation}
	Let $s = \deg_y(f_1) \geq 1$ and let us write $$ f_1(x, y) = \sum_{i=0}^{s} f_{1,i}(x) y^i,$$ where $ f_{1,i}(x) \in \k[x]$.
	By looking at the coefficient of  degree $s-1$ with respect to $y$ in both sides of \eqref{conditioncommefromxx}
	we obtain $$u^{s-1} f_{1,s-1}(ex)+su^{s-1} f_{1,s}(ex)h(x)=e^{kn+1}f_{1,s-1}(x).$$ Since $f_{1,s}(x)\neq0$ we have $\deg(h)\leq \deg_{x}(f_{1,s-1})\leq \deg_{x}(f_{1})$.
	
	\medskip -- The case $f_1(x,y)=f_1(x)\in\k[x]$ and $s=\deg_y(f_2)\geq 2$. From the equation (\ref{secondCondition}) we have 
	$$f_2(e\bar{x},u\bar{y}+h(\bar{x}))=e^{kn}u f_2(\bar{x},\bar{y}) + e^{kn}h^{\prime}(\bar{x})f_1(\bar{x})$$
	and since $\k[\bar{x},\bar{y}]=\k^{[2]}$ we actually have \begin{equation}\label{conditioncommefromyy} f_2(ex,uy+h(x))=e^{kn}u f_2(x, y)+e^{kn} h^{\prime}(x)f_1(x).\end{equation}
	Let us write $$ f_2(x, y) = \sum_{i=0}^{s} f_{2,i}(x) y^i,
	$$ where $ f_{2,i}(x) \in \k[x]$. By looking at the coefficient of  degree $s-1\geq 1$ with respect to $y$ in both sides of \eqref{conditioncommefromyy}
	we obtain $$u^{s-1} f_{2,s-1}(ex)+su^{s-1} f_{2,s}(ex)h(x)=e^{kn}uf_{2,s-1}(x).$$ Since $f_{2,s}(x)\neq0$, we have $\deg(h)\leq \deg_{x}(f_{2,s-1})\leq \deg_{x}(f_{2})$.
	
	\medskip -- The case $f_1(x,y)=f_1(x)\in\k[x]$ and $s=\deg_y(f_2)=1$. To simplify, we let $\xi$ stand for $\xi_{c,q}$ and recall that $\deg_{\xi}$ denotes the degree function associated to $\xi$. The assumption that $\deg_y(f_2)=1$ implies that $\deg_{\xi}(\delta(\bar{y}))=1$ and from Proposition \ref{xbasisProposition} it follows that \begin{equation}\label{equationThirdCase}\delta(\bar{y})=a(\bar{x})\bar{y}+b(\bar{x})\end{equation} where $a(x),b(x)\in\k[x]$ and $a(x)\neq 0$. On the other hand, since $c(\bar{x})^k\delta(\bar{x})=f_1(\bar{x})\in\k[\bar{x}]$ and $c(\bar{x})\neq 0$ and $\k[\bar{x}]$ is factorially closed in $\a_{c,q}$ it follows that $\delta(\bar{x})=g(\bar{x})\in\k[\bar{x}]$.
	
	\medskip By applying $\sigma$ to both sides of (\ref{equationThirdCase}), and taking into account $\sigma\delta=\delta\sigma$ and Proposition \ref{xbasisProposition} and $\k[\bar{x},\bar{y}]=\k^{[2]}$ we obtain $a(ex)=a(x)$ and \begin{equation}\label{equationThirdCaseTwo}a(ex)h(x)-g(x)h^{\prime}(x)=ub(x)-b(ex).
	\end{equation} 
	If we have $$\deg\big(a(ex)h(x)-g(x)h^{\prime}(x)\big)=\max\big(\deg(a(ex)h(x)),\deg(g(x)h^{\prime}(x))\big)$$ then from (\ref{equationThirdCaseTwo}) it follows in particular that $\deg(h)\leq \deg(b)$ and we are done. Otherwise, $g(x)$ is nonzero and the leading coefficients of $a(ex)h(x)$ and $g(x)h^{\prime}(x)$ cancel out. Thus, if we let $\lc(a)$ and $\lc(g)$ be respectively the leading coefficients of $a(x)$ and $g(x)$ then $\deg(h)=\lc(a)\lc(g)^{-1}$, which proves that $\deg(h)$ depends only on $\delta$ and not on the automorphism $\sigma\in\auto_{\k}(\a_{c,q},\delta)$.
	
	\medskip -- The case $f_1(x,y)=f_1(x)\in\k[x]$ and $f_2(x,y)=f_2(x)\in\k[x]$. From the assumption that $\delta$ is not locally nilpotent we deduce that $f_1(x)\neq 0$. On the other hand, using similar arguments as in the first three cases we obtain $$f_2(ex)=e^{kn}uf_2(x)+h^{\prime}(x)f_1(x)$$ and hence $\deg(h)\leq \deg(f_2)+1$.
\end{proof}

\begin{lemma}\label{derivationShape}Let $\delta$ be a $\k$-derivation of $\a_{c,q}$ such that $\auto_{\k}(\a_{c,q},\delta)\cap \saut_{\k}(\a_{c,q})$ is nontrivial. Then $$\delta(\bar{x})=g(\bar{x})\quad \mathrm{and}\quad \delta(\bar{y})=a(\bar{x})\bar{y}+b(\bar{x})$$ for some $g(x),a(x),b(x)\in\k[x]$.
\end{lemma}
\begin{proof}Let $\sigma\in \auto_{\k}(\a_{c,q},\delta)\cap \saut_{\k}(\a_{c,q})$ be a nontrivial automorphism and let us write $\sigma=\exp(\eta)$ where, by Theorem \ref{basisLemma}, $\eta=\ell(\bar{x})\xi_{c,q}$ for some nonzero $\ell(x)\in\k[x]$. The assumption that $\sigma$ and $\delta$ commute implies that $\eta\delta=\delta\eta$, and hence $\ker(\eta)=\k[\bar{x}]$ is stable under $\delta$. This proves that $\delta(\bar{x})=g(\bar{x})$ for some $g(x)\in\k[x]$. On the other hand, since $c(\bar{x})=\xi_{c,q}(\bar{y})$ generates the plinth ideal $\pl(\xi_{c,q})$ it follows that $\pl(\eta)$ is generated by $\ell(\bar{x})c(\bar{x})=\eta(\bar{y})$. Since moreover $\eta\delta(\bar{y})=\delta(\eta(\bar{y}))\in\k[\bar{x}]$ it follows that $\delta(\bar{y})=a(\bar{x})\bar{y}+b(\bar{x})$ for some $a(x), b(x)\in\k[x]$.
\end{proof}

\begin{proof}[Proof of Theorem \ref{isotropyGroupTheorem}] By Lemma \ref{closedSubgroup}, $\auto_{\k}(\a_{c,q},\delta)$ is a closed algebraic subgroup of the ind-group $\auto_{\k}(\a_{c,q})$. Let $\psi: \auto_{\k}(\a_{c,q})\longrightarrow\T_2$ be the canonical homomorphism, and recall that, by Theorem \ref{splittingTheoremCentered}, $\image(\psi)=\G_{c,q}$ and, by Lemma \ref{nestedStructure}, $\psi$ is an ind-group homomorphism. Therefore the restriction of $\psi$ to $\auto_{\k}(\a_{c,q},\delta)$, say $\vartheta$, is an algebraic group homomorphism. Moreover, $\image(\vartheta)$ is a closed algebraic subgroup of $\G_{c,q}$ and $\ker(\vartheta)=\auto_{\k}(\a_{c,q},\delta)\cap \saut_{\k}(\a_{c,q})$ is a closed algebraic subgroup of $\saut_{\k}(\a_{c,q})$. Thus, we have one of the following two cases.
	
	\medskip -- The homomorphism $\vartheta$ is injective. In this case, $\vartheta$ induces an algebraic group isomorphism from $\auto_{\k}(\a_{c,q},\delta)$ onto a closed algebraic subgroup of $\G_{c,q}$. Since the dimension of $\G_{c,q}$ is a most two it follows that $\auto_{\k}(\a_{c,q},\delta)$ is a most two-dimensional. 
	
	\medskip -- The morphism $\vartheta$ is not injective. We will show in this case that the exact sequence
	$$\xymatrix{1 \ar[r] & \ker(\vartheta) \ar[r]^-{} & \auto_{\k}(\a_{c,q},\delta) \ar[r]^-{\vartheta} & \image(\vartheta) \ar[r] & 1}$$ is split and $\ker(\vartheta)$, as a closed algebraic subgroup of $\auto_{\k}(\a_{c,q},\delta)$, is isomorphic to $\G_a=(\k,+)$.
	
	\medskip From Lemma \ref{derivationShape} it follows that there exist $g(x),a(x),b(x)\in\k[x]$ such that $$\delta(\bar{x})=g(\bar{x})\quad \mathrm{and}\quad \delta(\bar{y})=a(\bar{x})\bar{y}+b(\bar{x}).$$ Since $\delta$ is not locally nilpotent it follows that either $a(x)\neq 0$ or $g(x)\neq 0$. Now let $$L=\big\{\ell(x)\in\k[x]\mid \exp(\ell(\bar{x})\xi_{c,q})\in\ker(\vartheta) \big\}.$$ Then $L$ is a nonzero $\k$-vector subspace of $\k[x]$. On the other hand, for every $\ell(x)\in\k[x]$ the automorphism $\sigma=\exp(\ell(\bar{x})\xi_{c,q})$ commutes with $\delta$ if and only if $\sigma\delta(\bar{y})=\delta\sigma(\bar{y})$. Taking into account the fact that $\k[\bar{x},\bar{y}]=\k^{[2]}$, this is equivalent to \begin{equation}\label{homogeneousEquation}
		a(x)h(x)-g(x)h^{\prime}(x)=0,
	\end{equation}
	where $h(x)=c(x)\ell(x)$. Therefore, $\ell(x)\in L$ if and only if $h(x)=c(x)\ell(x)$ satisfies the equation (\ref{homogeneousEquation}). The fact that $L$ is nonzero, and $a(x)$ and $g(x)$ are not both zero then implies that $g(x)\neq 0$. Let $\ell_1(x),\ell_2(x)\in L$ be nonzero and let $h_i(x)=c(x)\ell_i(x)$. Then we have $$\frac{h_1^{\prime}(x)}{h_1(x)}=\frac{h_2^{\prime}(x)}{h_2(x)}=\frac{a(x)}{g(x)},$$ and then $$\left(\frac{h_1(x)}{h_2(x)}\right)^{\prime}=0.$$ This proves that $h_2(x)=\mu h_1(x)$ for some $\mu\in\k$, and hence the $\k$-vector space $L$ is one-dimensional. Let us fix a nonzero polynomial $\ell_0(x)\in L$ and write $h_0(x)=c(x)\ell_0(x)$, and let $v\in\k$ be such that $h_0(v)\neq 0$. The map $$\ell(x)\in L\longmapsto c(v)\ell(v)\in \k$$ is then an isomorphism of $\k$-vector spaces, and hence the map $$\sigma=\exp(\ell(\bar{x})\xi_{c,q})\in\ker(\vartheta)\longmapsto c(v)\ell(v)\in \k$$ is an algebraic group isomorphism from the closed algebraic subgroup $\ker(\vartheta)$ of $\auto_{\k}(\a_{c,q},\delta)$ onto $\G_a$.
	
	\medskip Let $\sigma\in\auto_{\k}(\a_{c,q})$, and write $\sigma(\bar{x})=e\bar{x}$ and $\sigma(\bar{y})=u\bar{y}+h(\bar{x})$. Then $\sigma\delta=\delta\sigma$ if and only if $\sigma\delta(\bar{x})=\delta\sigma(\bar{x})$ and $\sigma\delta(\bar{y})=\delta\sigma(\bar{y})$. A direct computation shows that this is equivalent to $g(ex)=eg(x)$, $a(ex)=a(x)$ and
	\begin{equation}\label{nonHomogeneousEquation}a(x)h(x)-g(x)h^{\prime}(x)=ub(x)-b(ex).
	\end{equation}
	The solutions of the equation (\ref{nonHomogeneousEquation}) have the form $h(x)=h_1(x)+\mu h_0(x)$, where $h_1(x)$ is a particular solution of (\ref{nonHomogeneousEquation}) and $\mu\in\k$.
	
	\medskip Let $m$ be the multiplicity of $0$ as a root of the polynomial $h_0(x)$. Then there exists a unique $\ell_{e,u}(x)\in\k[x]$ such that $h_{e,u}(x)=c(x)\ell_{e,u}(x)$ satisfies the equation (\ref{nonHomogeneousEquation}) and $h_{e,u}^{(m)}(0)=0$, where $h_{e,u}^{(m)}(x)$ is the $m$-th derivative of $h_{e,u}(x)$. Let $\sigma_{e,u}\in\auto_{\k}(\a_{c,q},\delta)$ be the $\k$-automorphism defined by $\sigma_{e,u}(\bar{x})=e\bar{x}$ and $\sigma_{e,u}(\bar{y})=u\bar{y}+h_{e,u}(\bar{x})$, and let us show that the map $$\varrho:(e,u)\in\image(\vartheta)\longmapsto \sigma_{e,u}\in\auto_{\k}(\a_{c,q},\delta)$$ is a group homomorphism.
	
	\medskip Let $(e_1,u_1),(e_2,u_2)\in\image(\vartheta)$ and to simplify let us write $\varrho(e_i,u_i)=\sigma_i$ and $\sigma_i(\bar{y})=u_i\bar{y}+h_i(\bar{x})$. A direct computation shows that $\sigma_1\sigma_2(\bar{x})=e_1e_2\bar{x}$ and $$\sigma_1\sigma_2(\bar{y})=u_1u_2\bar{y}+u_2h_1(\bar{x})+h_2(e_1\bar{x}).$$ Since $\sigma_1\sigma_2\in\auto_{\k}(\a_{c,q},\delta)$ it follows that $u_2h_1(\bar{x})+h_2(e_1\bar{x})$ satisfies the equation (\ref{nonHomogeneousEquation}) corresponding to $(e,u)=(e_1e_2,u_1u_2)$. On the other hand, we have $$\big(u_2h_1(x)+h_2(e_1x)\big)^{(m)}=u_2h_2^{(m)}(x)+e_1^mh_1^{(m)}(e_1x),$$ which clearly shows that $$\big(u_2h_1(x)+h_2(e_1x)\big)^{(m)}(0)=0.$$ Therefore, $\varrho(e_1e_2,u_1u_2)=\sigma_1\sigma_2=\varrho(e_1,u_1)\varrho(e_2,u_2)$. This finally proves that $\auto_{\k}(\a_{c,q},\delta)$ is the semi-direct product of $\ker(\vartheta)\cong \G_a$ and $\image(\vartheta)$ which is a closed algebraic subgroup of $\G_{c,q}$. As a consequence, the dimension of $\auto_{\k}(\a_{c,q},\delta)$ is at most three.
\end{proof}

The following is a typical example where $\auto_{\k}(\a_{c,q},\delta)$ has dimension three.
\begin{example}Let $c(x)=x^n$ and $q(x,y)=y^d$, where $n,d\geq 2$, and consider $$\a_{c,q}=\k[x,y,z]/(x^nz-y^d).$$ Clearly, we have $\G_{c,q}=\T_2$ the two-dimensional algebraic torus over $\k$.
	
	\medskip Let $a>n$ be an integer and $b\in\k$, and consider the $\k$-derivation $\Delta$ of $\k[x,y,z]$ defined by $$\Delta(x)=x,\quad \Delta(y)=ay+bx^n,\quad \Delta(z)=(ad-n)z+dby^{d-1}.$$ A direct computation shows that $\Delta(x^nz-y^d)=ad(x^nz-y^d)$, and hence $\Delta$ induces a $\k$-derivation $\delta$ of $\a_{c,q}$. The equation (\ref{homogeneousEquation}) writes in this case as $$ah(x)-xh^{\prime}(x)=0,$$ and its general solution is $\mu x^a$ where $\mu$ ranges over $\k$. Since $a>n$ it follows that $c(x)=x^n$ divides $\mu x^a$, and hence $\auto_{\k}(\a_{c,q},\delta)\cap \saut_{\k}(\a_{c,q})$ consists of the automorphisms $\exp(\mu\bar{x}^{a-n}\xi_{c,q})$. On the other hand, for every $(e,u)\in \G_{c,q}$ the equation (\ref{nonHomogeneousEquation}) writes as $$ah(x)-xh^{\prime}(x)=(u-e^n)bx^n,$$ and a particular solution of this equation is $h_{e,u}(x)=\frac{u-e^n}{a-n}bx^n$. Therefore, the general solution of (\ref{nonHomogeneousEquation}) is $h_{e,u}(x)+\mu x^a$, where $\mu\in\k$. Since the multiplicity of $0$ as a root of $x^a$ is $a$ it follows that the unique solution of (\ref{nonHomogeneousEquation}) that satisfies $h^{(a)}(0)=0$ is $h_{e,u}(x)$. Thus, if we let $\sigma_{e,u}$ be the $\k$-automorphism of $\a_{c,q}$ given by $\sigma(\bar{x})=e\bar{x}$ and $\sigma(\bar{y})=u\bar{y}+\bar{x}^nh_{e,u}(\bar{x})$ then $\sigma_{e,u}$ commutes with $\delta$ and the map $$(e,u)\in\G_{c,q}\longmapsto \sigma_{e,u}\in \auto_{\k}(\a_{c,q},\delta)$$ is a morphism of algebraic groups. By Theorem \ref{isotropyGroupTheorem}, $\auto_{\k}(\a_{c,q},\delta)$ is the semi-direct product of a subgroup isomorphic to $\T_2$ and a subgroup isomorphic to $\G_a$. This proves in particular that $\auto_{\k}(\a_{c,q},\delta)$ is three-dimensional. It is also worth mentioning that if $b\neq 0$ then $h_{e,u}(x)$ is in general nonzero, and hence the group action underlying the semi-direct product structure of $\auto_{\k}(\a_{c,q},\delta)$ is not induced by the group action underlying the semi-direct product structure of $\auto_{\k}(\a_{c,q})$ given in Theorem \ref{splittingTheoremCentered}.
\end{example}

\bigskip \noindent {\bf Funding} The second author was supported by Rio Grande do Sul Research Foundation (FAPERGS, project number: 82451.812.41312.27032024), and IMPA Postdoctoral Program, Brazil.

\smallskip \noindent The fourth author was supported by the RSF grant 25-21-00277.

\vspace{5mm}
\small

\ttfamily Abdessamad Ahouita, Department of Mathematics, Faculty of Sciences Semlalia, Cadi Ayyad University, Morocco.

\textit{E-mail address:} a.ahouita.ced@uca.ac.ma 
\vspace{3mm}

\ttfamily Rene Baltazar, Institute of Mathematics, Statistics and Physics, Federal University of Rio Grande, Brazil.

\textit{E-mail address:} renebaltazar.furg@gmail.com

\vspace{3mm}
\ttfamily M'hammed El Kahoui, Department of Mathematics, Faculty of Sciences Semlalia, Cadi Ayyad University, Morocco.

\textit{E-mail address:} elkahoui@uca.ac.ma
\vspace{3mm}

\ttfamily Sergey Gaifullin, Faculty of Mechanics and Mathematics, Department of Higher Algebra, Lomonosov Moscow State University, Russia and Faculty of Computer Science, HSE University, Russia.

\textit{E-mail address:} sgayf@yandex.ru


\begin{thebibliography}{plain}

\bibitem{baltazarPan_2021}
Baltazar, R. and Pan, I. On the automorphism group of a polynomial differential ring in two variables. \textit{Journal of Algebra}, 576:197--227, 2021.

\bibitem{baltazarVeloso2020}
Baltazar, R. and Veloso, M. On isotropy group of Danielewski surfaces. \textit{Communications in Algebra}, 49(3):1006--1016, 2021.

\bibitem{benkhaddah_elkahoui_ouali_2023}
Ben Khaddah, A., El Kahoui, M., and Ouali, M. The freeness property for locally nilpotent derivations of $R^{[2]}$. \textit{Transformation Groups}, 28(1):35--60, 2023.

\bibitem{bianchi_vesolo_2017}
Bianchi, A. C. and Veloso, M. O. Locally nilpotent derivations and automorphism groups of certain Danielewski surfaces. \textit{Journal of Algebra}, 469:96--108, 2017.

\bibitem{danielewski_89}
Danielewski, W. On a cancellation problem and automorphism groups of affine algebraic varieties. Preprint, Warsaw, 1989.

\bibitem{dasguptaLahiri_2023}
Dasgupta, N. and Lahiri, A. Isotropy subgroups of some almost rigid domains. \textit{Journal of Pure and Applied Algebra}, 227(4):107250, 2023.

\bibitem{elkahoui25}
El Kahoui, M., Essamaoui, N., and Ouali, M. The centralizer of a locally nilpotent $R$-derivation of the polynomial $R$-algebra in two variables. \textit{Journal of Pure and Applied Algebra}, 229(1):107828, 2025.

\bibitem{elkahouiHammi_2025}
El Kahoui, M. and Hammi, A. Automorphisms of simple Danielewski algebras. \textit{Journal of Algebra and its Applications}, to appear, 2024.

\bibitem{essen_book}
van den Essen, A. \textit{Polynomial automorphisms and the Jacobian conjecture}. Progress in Mathematics, Vol. 190, Birkhäuser Verlag, Basel, 2000.

\bibitem{finstonWalcher97}
Finston, D. R. and Walcher, S. Centralizers of locally nilpotent derivations. \textit{Journal of Pure and Applied Algebra}, 120(1):39--49, 1997.

\bibitem{freudenburg_book}
Freudenburg, G. \textit{Algebraic theory of locally nilpotent derivations}. Encyclopaedia of Mathematical Sciences, Vol. 136, 2nd ed., Springer-Verlag, Berlin, 2017.

\bibitem{freudenburg_2019}
Freudenburg, G. \emph{Canonical factorization of the quotient morphism for an affine $\mathbb{G}_a$-variety}, \textit{Transformation Groups}, vol.~24, no.~2, pp.~355--377, 2019.

\bibitem{furterKraft_2018}
Furter, J.-P. and Kraft, H. On the geometry of the automorphism groups of affine varieties. Preprint, arXiv:1809.04175 [math.AG], 2018.

\bibitem{makar-limanov_2001}
Makar-Limanov, L. On the group of automorphisms of a surface $x^n y = P(z)$. \textit{Israel Journal of Mathematics}, 121:113--123, 2001.

\bibitem{miyanishi_book}
Miyanishi, M. \textit{Curves on rational and unirational surfaces}. Tata Institute of Fundamental Research Lectures on Mathematics and Physics, Vol. 60, Bombay, 1978.

\bibitem{pan_2022}
Pan, I. A characterization of local nilpotence for dimension two polynomial derivations. \textit{Communications in Algebra}, 50(5):1884--1888, 2022.

\bibitem{poloni_2011}
Poloni, P.-M. Classification(s) of Danielewski hypersurfaces. \textit{Transformation Groups}, 16(2):579--597, 2011.

\bibitem{Rentschler_68}
Rentschler, R. Opérations du groupe additif sur le plan affine. \textit{C. R. Acad. Sci. Paris Sér. A–B}, 267:A384--A387, 1968.

\bibitem{shafarevich_1966}
Shafarevich, I. R. On some infinite-dimensional groups. \textit{Rendiconti di Matematica e delle sue Applicazioni}, 25(1–2):208--212, 1966.

\bibitem{stampfli_2014}
Stampfli, I. Automorphisms of $\mathbb{C}^3$ commuting with a $\mathbb{C}^+$-action. \textit{International Mathematics Research Notices (IMRN)}, 2015(19):9832--9856.

\bibitem{zaidenberg_2017}
Kovalenko, S., Perepechko, A., and Zaidenberg, M. On automorphism groups of affine surfaces. In \textit{Algebraic varieties and automorphism groups}, Adv. Stud. Pure Math. 75, Math. Soc. Japan, Tokyo, 2017.

\end{thebibliography}

\begin{bibdiv}
\begin{biblist}
\bib{Balta}{article}{
title = {On simple Shamsuddin derivations in two variables},
journal = {Annals of the Brazilian Academy of Sciences},
volume = {88 (4)},
pages = {2031-2038},
year = {2016},
issn = {},
doi = {},
url = {},
author = {R. Baltazar},
keywords = {},
abstract = {}
} 

\bib{PanMendes}{article}{
title = {On plane polynomial automorphisms commuting with simple derivations},
journal = {Journal of Pure and Applied Algebra},
volume = {221 4},
pages = {875 - 882},
year = {2017},
issn = {},
doi = {},
url = {},
author = {I. Pan and L.G. Mendes},
keywords = {},
abstract = {}
}

\bib{LevBert}{article}{
title = {On the isotropy group of a simple derivation},
journal = {Journal of Pure and Applied Algebra},
volume = {224},
pages = {33-41},
year = {2020},
issn = {},
doi = {},
url = {},
author = {L.N. Bertoncello and D. Levcovitz},
keywords = {},
abstract = {}
}


\bib{HuYan}{article}{
title = {Simple derivations and the isotropy groups},
journal = {Communications in Algebra},
volume = {53},
pages = {1097-1112},
year = {2024},
issn = {},
doi = {},
url = {},
author = {Y. Huang and D. Yan},
keywords = {},
abstract = {}
}

\bib{BaltaVeloso}{article}{
title = {On isotropy group of Danielewski surfaces},
journal = {Communications in Algebra},
volume = {49},
pages = {1006-1016},
year = {2021},
issn = {},
doi = {},
url = {},
author = {R. Baltazar and M. Veloso},
keywords = {},
abstract = {}
}

\bib{Dasgupta}{article}{
title = {Isotropy subgroups of some almost rigid domains},
journal = {Journal of Pure and Applied Algebra},
volume = {227 4},
pages = {107250},
year = {2023},
issn = {},
doi = {},
url = {},
author = {N. Dasgupta and A.Lahiri },
keywords = {},
abstract = {}
}

\bib{Rittatore}{article}{
title = {On the orbits of plane automorphisms
and their stabilizers},
journal = {New York Journal of Mathematics},
volume = {29},
pages = {1373–1392},
year = {2023},
issn = {},
doi = {},
url = {},
author = {A. Rittatore and I. Pan},
keywords = {},
abstract = {}
}

\bib{Kour}{article}{
title = {Isotropy group of non-simple derivations of K[x, y]},
journal = {Communications in Algebra},
volume = {51},
pages = {5065-5083},
year = {2023},
issn = {},
doi = {},
url = {},
author = {H. Rewri and S. Kour},
keywords = {},
abstract = {}
}

\bib{vivas}{article}{
title = {Automorphisms and isomorphisms of quantum generalized Weyl algebras},
journal = {Journal of Algebra},
volume = {424},
pages = {540-552},
year = {2015},
issn = {0021-8693},
doi = {https://doi.org/10.1016/j.jalgebra.2014.08.045},
url = {https://www.sciencedirect.com/science/article/pii/S0021869314005122},
author = {Mariano Suárez-Alvarez and Quimey Vivas},
keywords = {Generalized Weyl algebra, Automorphism group, Isomorphism problem},
abstract = {We classify up to isomorphism the quantum generalized Weyl algebras and determine their automorphism groups in all cases in a uniform way, including the case where the parameter q is a root of unity, thereby completing the results obtained by Bavula and Jordan (2001) [5] and Richard and Solotar (2006) [12].}
}
\bib{bre}{article}{
title = {Skew derivations on generalized Weyl algebras},
journal = {Journal of Algebra},
volume = {493},
pages = {194-235},
year = {2018},
issn = {0021-8693},
doi = {https://doi.org/10.1016/j.jalgebra.2017.09.018},
url = {https://www.sciencedirect.com/science/article/pii/S0021869317305124},
author = {Munerah Almulhem and Tomasz Brzeziński},
keywords = {Generalized Weyl algebra, Skew derivation},
abstract = {A wide class of skew derivations on degree-one generalized Weyl algebras R(a,φ) over a ring R is constructed. All these derivations are twisted by a degree-counting extensions of automorphisms of R. It is determined which of the constructed derivations are Q-skew derivations. The compatibility of these skew derivations with the natural Z-grading of R(a,φ) is studied. Additional classes of skew derivations are constructed for generalized Weyl algebras given by an automorphism φ of a finite order. Conditions that the central element a that forms part of the structure of R(a,φ) needs to satisfy for the orthogonality of pairs of aforementioned skew derivations are derived. In addition local nilpotency of constructed derivations is studied. General constructions are illustrated by description of all skew derivations (twisted by a degree-counting extension of the identity automorphism) of generalized Weyl algebras over the polynomial ring in one variable and with a linear polynomial as the central element.}
}

\bib{Limanov}{article}{
title = {On automorphisms of Weyl algebra},
journal = {Bull. Soc. Math. France},
volume = {112},
pages = {359–363},
year = {1984},
issn = {},
doi = {},
url = {},
author = {L. Makar-Limanov},
keywords = {},
abstract = {}
}

\bib{Dixmier}{article}{
title = {Sur les algebres de Weyl},
journal = {Bull. Soc. Math. France},
volume = {96},
pages = {209–242},
year = {1968},
issn = {},
doi = {},
url = {},
author = {J. Dixmier},
keywords = {},
abstract = {}
}

\bib{Shirikov}{article}{
title = {Two-generated graded algebras},
journal = {Algebra Discrete Math.},
volume = {3},
pages = {64–80},
year = {2005},
issn = {},
doi = {},
url = {},
author = {E. N. Shirikov},
keywords = {},
abstract = {}
}

\bib{Limanov}{article}{
title = {Derivations of Ore extensions of the polynomial ring in one variable}
journal = {Comm. Algebra},
volume = {9},
pages = {3651–3672},
year = {2004},
issn = {},
doi = {},
url = {},
author = {A. Nowicki},
keywords = {},
abstract = {}
}

\bib{Samuel}{article}{
title = {Derivations of a parametric family of subalgebras of the Weyl algebra}
journal = {Journal of Algebra},
volume = {424},
pages = {46-97},
year = {2015},
issn = {},
doi = {},
url = {},
author = {G.Benkart, S.Lopes, and M. Ondrus},
keywords = {},
abstract = {}
}



\bib{McConnelPetit}{article}{
title = {Crossed Products and Multiplicative Analogues of Weyl Algebras},
journal = {Journal of the  London Mathematical Society},
volume = {s2-38(1)},
pages = {47 - 55},
year = {1988},
issn = {},
doi = {},
url = {},
author = {McConnell, J. C., and Pettit, J. J..},
keywords = {},
abstract = {}
}

\bib{art}{article}{
title = {Generalized Derivations of a Quantum Plane}
journal = {Journal of Mathematical Sciences},
volume = {131, no 5},
pages = {5904-5918},
year = {2005},
issn = {},
doi = {},
url = {},
author = {Artamanov V. A.},
keywords = {},
abstract = {}
}
\end{biblist}
\end{bibdiv}
\end{document}